\documentclass[12pt,leqno]{article}
\usepackage{graphicx, fullpage}
\usepackage{amsmath,amssymb,amsthm,amscd, bm}
\usepackage{fancyhdr}
\usepackage[mathscr]{eucal}
\usepackage{amsfonts}
\begin{document}
\parskip=6pt

\theoremstyle{plain}

\newtheorem {thm}{Theorem}[section]
\newtheorem {lem}[thm]{Lemma}
\newtheorem {cor}[thm]{Corollary}
\newtheorem {defn}[thm]{Definition}
\newtheorem {prop}[thm]{Proposition}
\numberwithin{equation}{section}
\def\cal{\mathcal}
\newcommand{\cF}{\cal F}
\newcommand{\cA}{\cal A}
\newcommand{\cC}{\cal C}
\newcommand{\cO}{{\cal O}}
\newcommand{\cE}{{\cal E}}
\newcommand{\cU}{{\cal U}}
\newcommand{\cM}{{\cal M}}
\newcommand{\cD}{{\cal D}}
\newcommand{\bC}{\mathbb C}
\newcommand{\bP}{\mathbb P}
\newcommand{\bN}{\mathbb N}
\newcommand{\bA}{\mathbb A}
\newcommand{\bR}{\mathbb R}
\newcommand{\var}{\varepsilon}
\newcommand{\End}{\text{End }}
\newcommand{\loc}{\text{loc}}
\newcommand{\dist}{\text{dist}}
\newcommand{\rk}{\roman{rk }}
\renewcommand\qed{ }
\begin{titlepage}
\title{\bf Modules of square integrable holomorphic germs\thanks{Research partially supported by NSF grant DMS-1162070\newline
2010 Mathematics subject classification 14F18, 32C35, 32L20, 32U05}}
\author{L\'aszl\'o Lempert\\ Department of  Mathematics\\
Purdue University\\West Lafayette, IN
47907-2067, USA}
\thispagestyle{empty}
\end{titlepage}
\date{}
\maketitle
\abstract
This paper was inspired by Guan and Zhou's recent proof of the so-called strong openness conjecture for plurisubharmonic functions.
We give a proof shorter than theirs and extend the result to possibly singular hermitian metrics on vector bundles.
\endabstract

\section{Introduction}

Consider an open set $U\subset\bC^m$ and a point $x\in U$.
Given a holomorphic function $f$ on some neighborhood $V$ of $x$, we will denote by $\mathbf{f}_x$ its germ at $x$.
A measurable function $u\colon U\to [-\infty,\infty]$ determines an ideal $I(u)=I(u,x)$ in the ring $\cO_x=\cO_{(\bC^m,x)}$ of holomorphic germs at $x$
$$
I(u)=\{ \mathbf{f}_x\colon f\in \cO(V),\ \int_V |f|^2 e^{-u} < \infty,\ V\subset U\text{ open}, x\in V\}.
$$
The integral is with respect to Lebesgue measure in $\bC^m$.
Clearly, if $v\leq u+O(1)$ at $x$, then $I(u)\supset I(v)$.
A conjecture, going back to Demailly and Koll\'ar [DK,D2] had that if $u$ is plurisubharmonic, then $I(u)=I(\eta u)$ with some $\eta\in
(1,\infty)$.
After partial results by Favre-Jonsson and Berndtsson, Guan and Zhou recently posted a proof of the conjecture, see [B3,FJ,GZ2-4].
A related posting is [Hi].

Our main purpose with this paper is to produce a proof, as we hope more transparent than Guan-Zhou's, by modifying their approach some, while
keeping all their essential ideas.
We will also discuss generalizations.
The most immediate generalization replaces multiples of $u$ by a sequence of plurisubharmonic functions:

\begin{thm} If $u_1\leq u_2\leq \ldots$ are plurisubharmonic functions on $U$ and $u=\lim_j u_j$ is locally bounded above on $U$, then
$I(u)=I(u_j)$ for some $j$.
\end{thm}

Since in the original conjecture one can always assume that $u$ is bounded above, and then in fact that $u\leq 0$, the conjecture indeed follows from
Theorem 1.1 if one puts $u_j=(1+1/j)u$.---That the ideas of Guan and Zhou also give a proof of Theorem 1.1 occurred to me while reading [GZ2].
Apparently at one point Guan and Zhou also noticed this, because after I communicated to them the generalization, Zhou sent me a preprint
containing essentially the same result, dated earlier than my email to them, in fact even earlier than the submission date of [GZ2].
Subsequently this generalization was also mentioned in [GZ3]. A variant, in dimension 2, occurs already in [FJ, Proposition 2.6].

A further generalization involves, instead of ideals, modules of square integrable vector valued holomorphic germs.
The natural setting here is germs of holomorphic sections of a holomorphic Hilbert bundle $E\to U$, and the role of the weight $e^{-u}$ is
played by a possibly singular hermitian metric $h$ on $E$.
The precise meaning of this and related notions will
be explained in Section 4.
For the time being, let $h\colon E\to [0,\infty]$ be any Borel measurable function on $E$.
Indicating the space of holomorphic sections of a vector bundle by $\Gamma$, we are led to consider the sets 
\begin{equation}
E(h,x)=\{\mathbf{f}_x\colon f\in\Gamma (V,E),\int_V h(f) <\infty,\ V\subset U\text{ open}, x\in V\},
\end{equation}
where $\mathbf{f}_x$ again stands for germ at $x$.
This $E(h,x)$ is an $\cO_x$--module for example if $\sqrt{h}$ is subadditive on the fibers of $E$ and homogeneous in the sense that $\sqrt{h(\lambda e)}=|\lambda| \sqrt{h(e})$ for
$\lambda\in\bC$ and $e\in E$.
Assuming $h$ has this property, as we let $x$ vary, the modules $E(h,x)$ form the stalks of a sheaf of modules denoted $\cE(h)$, a subsheaf of
the sheaf of holomorphic sections of $E$.

\begin{thm}Suppose $E\to U$ is a holomorphic Hilbert bundle and $h_1\geq h_2\geq\ldots$ are hermitian metrics on $E$ whose Nakano curvatures 
dominate $0$. Suppose that $h=\lim_j h_j$ is bounded below by a continuous hermitian metric.
If $\textrm{rk } E<\infty$, or at least $\bigcup_j \cE(h_j)$ is locally finitely generated, then
$\bigcup_j\cE (h_j)=\cE (h)$.
\end{thm}

In Section 7 we will see that such a result fails if we simply drop the assumption of finite generation.
Yet it seems to be an interesting problem to find weaker conditions on $h_j$ to guarantee the conclusion of the theorem.
Certain direct images of positively curved line bundles provide examples of Hilbert bundles and hermitian metrics as in Theorem 1.2, see
\cite{B2}, and while in these direct images $\cE(h_j)$ and $\bigcup_j\cE(h_j)$ are locally infinitely generated, the strong openness theorem of
Guan--Zhou, or more generally Theorem 1.1 above, does provide a connection between $\bigcup_j\cE(h_j)$ and $\cE(h)$.

When $E$ is of finite rank, $\cE(h)$ is automatically locally finitely generated (= coherent, in this case), although we will not write out a proof.
So the conclusion of
Theorem 1.2 can be stated as $E(h,x)=E(h_j,x)$ for some $j$, in parallel with Theorem 1.1.
Thus Theorem 1.1 is a special case of Theorem 1.2, and it would suffice to prove the latter.
However, we will start by writing out the proof of the special case.
The proof of Theorem 1.2 will follow the same line, but it will be burdened by auxiliary material that is not as readily available for vector
bundles as for line bundles and that we will have to develop.

I am grateful to Bo Berndtsson for helpful discussions concerning \cite{GZ2} and \cite{C}, and to  Henri Skoda and Bernard Teissier  for
bibliographical information.

\section{The proof of Theorem 1.1}
In the setting of Theorem 1.1 we put $J=\bigcup_j I(u_j)$.
Suppose $P\subset\bC^m$ is a complex (affine) hyperplane and $W\subset P$ is relatively open.
For a measurable $g\colon W\to\bC$ define $||g||\in [0,\infty]$ by
$$
\|g\|^2=\inf_j\int_W |g|^2 e^{-u_j},
$$
the integral with respect to $2m-2$ dimensional Lebesgue measure.
By the dominated convergence theorem
\begin{equation}
\|g\|^2=\begin{cases}
\infty\qquad\text{ or}\\
\lim_j \int_W |g|^2 e^{-u_j}=\int_W |g|^2 e^{-u}.\end{cases}
\end{equation}
We denote by dist$(x,P)$ the Euclidean distance between $x$ and $P$, and write $P\|P_0$ to indicate that hyperplanes $P,P_0$ are parallel.
The crux of the matter is the following characterization of $J$ when $m\geq 1$:

\begin{lem}
Let $f\in\cO(U)$.
Its germ $\mathbf{f}_x$ is in $J$ if and only if for any sufficiently small neighborhood $V\subset U$ of $x$ and any hyperplane $P_0\subset\bC^m$
\begin{equation}
\liminf {\rm{dist}} (x,P) \|f|V\cap P\|=0,\quad\text{as } P\|P_0\text{ and }{\rm{dist}} (x,P)\to 0.
\end{equation}
\end{lem}

The lemma is of interest even when each $u_j=u$ (in fact, once Theorem 1.1 is proved, this is the only interesting case).
Indeed, assume $V$ is a polydisc centered at $x=0$ and $P_0=\{z_1=0\}$.
One direction of the lemma then says that, modulo shrinking, $\int_V |f|^2 e^{-u}<\infty$ provided
\[
\liminf_{s\to 0} |s|^{-2} \int_{V\cap \{z_1=s\}} |f|^2 e^{-u}=0.
\]
Since
\[
\int_V |f|^2 e^{-u}=\int_{\bC} \Big(\int_{V\cap \{z_1=s\}} |f|^2 e^{-u}\Big) d\lambda_2 (s),
\]
for certain type of functions $\varphi(s)$ the lemma provides the convergence of the integral $\int_{ \{|s|<r\}}\varphi(s) d\lambda_2 (s)$ once $\liminf_{s\to 0}\varphi(s)/|s|^2=0$ is known, a rather perplexing connection.

For the proof we need the following simple result, a variant of [GZ2, Lemma 2.3].
Let $\Delta\subset\bC$ denote the unit disc.

\begin{prop}Let $k\in\bN$ and let $F$ be a holomorphic function in a neighborhood of $\overline\Delta$, which does not vanish on $\overline\Delta\backslash \{0\}$.
Suppose $G\in\cO(\Delta)$, $G=o(F)$ at $0$, and with some $t\in\Delta\backslash \{0\}$
\[
F(\omega t)=G(\omega t)\text{ for all }k\text{'th roots of unity }\omega.
\]
Then
\begin{equation}
\sup_{\Delta} |G|\geq C_1 |t|^{-k},\quad C_1=\min_{|s|=1} |F(s)|>0.
\end{equation}
\end{prop}

\begin{proof}As in [GZ2] we start by writing $F(s)=s^p F_1(s)$, where $p\in\bN\cup \{0\}$ and $F_1$ does not vanish on $\overline\Delta$.
Upon dividing $F$ and $G$ by $F_1$ we reduce the proof to the case when $F(s)=s^p$.
Consider the function
\[
G_1(s)=\frac1k\sum_{\omega} \omega^{-p} G(\omega s),
\]
the sum over $k$'th roots of unity.
Our $G_1$ has all the properties listed above for $G$, and in addition, its Taylor series contains only monomials $s^q$ for which $q-p>0$ is divisible by $k$.
In particular, $q\geq p+k$, and so $G_1(s)/s^{p+k}$ is holomorphic on $\Delta$.
Hence
\[
\sup_{s\in\Delta} |G_1 (s)|= \sup_{s\in\Delta} |G_1 (s)/s^{p+k}|\geq |G_1(t)/t^{p+k}|=|t|^{-k},
\]
and (2.3) follows.
\end{proof}

\begin{proof}[ Proof of Lemma 2.1]
We will assume $x=0$.
Suppose first that $\mathbf f_0\in J$, and choose $j$ and a neighborhood $V$ of 0 so that $\int_V |f|^2 e^{-u_j}<\infty$.
Given $P_0$, change coordinates to arrange that $P_0$ is parallel to the hyperplane $\{z\in\bC^m\colon z_1=0\}$.
By Fubini's theorem
\begin{equation}
\infty > \int_V |f|^2 e^{-u_j}=\int_{\bC} \Big( \int_{V\cap \{z_1=\sigma\}} |f|^2 e^{-u_j}\Big) d\lambda_2 (\sigma).
\end{equation}
Since $\int |\sigma|^{-2} d\lambda_2 (\sigma)$ is a divergent integral over any neighborhood of $0\in\bC$, (2.4) implies
\[
\liminf_{\sigma\to 0} |\sigma|^2 \int_{V\cap \{z_1=\sigma\}} |f|^2 e^{-u_j}=0,
\]
and (2.2) follows.

Conversely, we will show that if $\mathbf{f}_0\not\in J$ then, given any neighborhood $V$ of 0, with some hyperplane $P_0$
\begin{equation}
\liminf\text{ dist}(x,P) \|f|V\cap P\| > 0,\quad\text{as } P\|P_0\text{ and dist}(x,P)\to 0.
\end{equation}

Fix such $V$, that we can assume to be pseudoconvex and relatively compact in $U$.
If $\alpha\colon\Delta\to U$ is holomorphic, $\alpha(0)=0$, we write $J\circ\alpha$ for the pull back
\[
\{\mathbf{g}\circ\alpha\colon \mathbf{g}\in J\}\subset \cO_{(\bC,0)}.
\]
Now $\mathbf{f}_0\not\in J$ implies there is a nonzero $\alpha$ such that $\mathbf{f}_0\circ\alpha\not\in \cO_{(\bC,0)}J\circ\alpha$, see [LT, Th\'eor\`eme 2.1], by Lejeune--Jalabert and Teissier, or [GZ2, Remark 2.12].
(What matters here is that $J$ is integrally closed, i.e., if $g^1,\ldots,g^p,\psi$ are holomorphic functions in a neighborhood of $0\in U$, the germs $\mathbf{g}_0^i$ are in $J$, and $\psi=O(|g^1|+\ldots+ |g^p|)$ at 0, then $\boldsymbol\psi_0\in J$.
Teissier tells me that he and Lejeune--Jalabert most probably learned the result from Hironaka.)
We choose a hyperplane $P_0$ through $0\in\bC^m$ that does not contain $\alpha(\Delta)$.
Upon adjusting the coordinates in $\bC$ and in $\bC^m$ we can assume that $P_0=\{z\in\bC^m\colon z_1=0\}$, that $\alpha=(\alpha_1,\ldots,\alpha_m)$ is holomorphic in a neighborhood of $\overline\Delta$, that $F=f\circ\alpha\neq 0$ on $\overline\Delta\backslash \{0\}$,
\begin{equation}
\alpha_1 (s)=s^k,\qquad s\in\Delta,
\end{equation}
and finally $\alpha(\overline\Delta)\subset V$.
This latter implies that there are constants $C_2,C'_2$ such that for $g\in\cO(V)$
\begin{equation}
\max_{\alpha(\overline\Delta)} |g|^2\leq C'_2 \int_V |g|^2\leq C_2^2 \int_V |g|^2 e^{-u}\leq C_2^2 \int_V |g|^2 e^{-u_j}
\end{equation}
for any $j$.
Write $P_\sigma$ for the hyperplane $\{z\in\bC^m\colon z_1=\sigma\}$.
We need to estimate $\|f|V\cap P_\sigma\|$ from below.
Take an arbitrary $\sigma\in\Delta\backslash \{0\}$.
Assume first that $\|f|V\cap P_\sigma\|<\infty$.
By the Ohsawa--Takegoshi theorem there is a $g\in\cO(V)$ that agrees with $f$ on $V\cap P_\sigma$ and with some $j$
\begin{equation}
\int_V |g|^2 e^{-u_j}\leq C_3^2 \|f|V\cap P_\sigma\|^2,
\end{equation}
where $C_3$ is independent of $j$ and $\sigma$.
Indeed, with some $j$
\[
\int_{V\cap P_\sigma} |f|^2 e^{-u_j}\leq 2\| f|V\cap P_\sigma\|^2,
\]
cf.~(2.1), and the Ohsawa--Takegoshi theorem, applied with this $u_j$, produces such a $g$.
Set $G=g\circ\alpha$, whose germ at $0\in\bC$ is in $J\circ\alpha$.
As $\mathbf{F}_0=\mathbf{f}_0\not\in \cO_{(\bC,0)}J\circ\alpha$, it must be that $G=o(F)$ at $0\in\bC$.
Further, by (2.6) $F(\root{k}\of \sigma)=G(\root{k}\of \sigma)$ for any choice of $k$th root $\root{k}\of \sigma$.
Hence Proposition 2.2 gives
\begin{equation}
\max_{\alpha(\overline\Delta)} |g|=\max_{\overline\Delta} |G| \geq C_1/|\sigma|.
\end{equation}
Putting together (2.8), (2.7), and (2.9)
\begin{equation}
\| f|V\cap P_\sigma\|\geq \frac {C_1}{ C_2 C_3 |\sigma|},\qquad \sigma\in\Delta\backslash \{0\}.
\end{equation}
This we derived under the assumption that the left hand side is finite, but of course (2.10) also holds when the left hand side is infinite.
(2.10) now implies (2.5), and the proof is complete.
\end{proof}

\begin{proof}[Proof of Theorem 1.1]
Since $I(u)\supset J$ and $I(u)$ is finitely generated, all we need to prove is that $\mathbf f_x\in I(u)$ implies $\mathbf f_x\in J$.
This we prove by induction on $m$, as in [GZ2].
The result is obvious when $m=0$; suppose it holds for $m-1$.
Upon shrinking $U$ we can assume $\mathbf f_x$ is the germ of some $f\in\cO (U)$.
First we apply the ``only if'' direction of Lemma 2.1, but with each $u_j$ replaced by $u$.
This provides a neighborhood $V_0\subset U$ of $x$ and for any hyperplane $P_0\subset\bC^m$ a sequence of hyperplanes $P_\nu\|P_0$ such that dist$(x,P_\nu)\to 0$ and
\begin{equation}
\lim_{\nu\to\infty}\text{ dist}(x,P_\nu)^2 \int_{V_0\cap P_\nu} |f|^2 e^{-u}=0.
\end{equation}

Let now $V$ be an arbitrary neighborhood of $x$, relatively compact in $V_0$.
Since $\int_{V_0\cap P_\nu}|f|^2 e^{-u}<\infty$ for $\nu>\nu_0$, for these $\nu$ the induction hypothesis gives a $j=j_\nu$ such that
\[
\int_{V\cap P_\nu} |f|^2 e^{-u_j}<\infty.
\]
Hence $\|f|V\cap P_\nu\|^2=\int_{V\cap P_\nu} |f|^2 e^{-u}$, cf.~(2.1).
Therefore (2.11) implies
\[
\lim_\nu\text{ dist}(x,P_\nu) \|f|V_0\cap P_\nu\|=0,
\]
and another application of Lemma 2.1, this time the ``if'' direction, proves $\mathbf f_x\in J$.
\end{proof}

\section{Smooth hermitian metrics and their curvature}

In this section we review the basics of hermitian metrics on holomorphic Hilbert bundles.
For the time being we restrict ourselves to trivial bundles $E=U\times H\to U$, with $U\subset\bC^m$ open and $(H,\langle ,\rangle)$ a complex Hilbert space.
We write End $H$ for the space of bounded linear operators on $H$, endowed with the operator norm.
If $k=0,1,\ldots,\infty$, a hermitian metric on $E$ of class $C^k$ is a function $h\colon E\oplus E\to\bC$ that can be represented as 
\begin{equation}
h\left( (z,\xi),(z,\eta)\right)=\langle P(z)\xi,\eta\rangle,\quad z\in U, \ \xi,\eta\in H,
\end{equation}
with $P\colon U\to\End H$ a $C^k$ map taking values in invertible, positive self adjoint operators.
If $e\in E$, we write $h(e)$ for $h(e,e)$.
Thus $\sqrt{h(e)}$ defines a norm on the fibers $E_z$ of $E$.
As usual, for two metrics $h\leq k$ means $h(e)\leq k(e)$ for all $e\in E$.

Just like in bundles of finite rank, a $C^2$ hermitian metric $h$ on $E$ has a curvature $R$, a $(1,1)$--form valued in End $E=\coprod_{z\in U}\End E_z$, see e.g.~[B2 or L3].
It is given by
\begin{equation}
R=\overline\partial (P^{-1}\partial P)=P^{-1} \overline\partial \partial P-P^{-1} \overline\partial P P^{-1}\wedge\partial P,
\end{equation}
where $P$ is from (3.1) and we identified $H=E_z$.
Curvature determines a hermitian form $N$ on each space $T_z^{1,0}U\otimes E_z=T_z^{1,0}U\otimes H$ (tensor product over $\bC$), given by
\begin{equation}
N(t\otimes\xi,u\otimes\eta)=h\left( R(t,\overline u)\xi,\eta\right),\quad t,u\in T_z^{1,0} U,\quad \xi,\eta\in E_z.
\end{equation}
Thus $N$ is a Hermitian form on $T^{1,0}U\otimes E$, that we call the Nakano curvature of $h$.
Instead of the usual terminology that $h$ has Nakano semipositive curvature we can then say that the Nakano curvature of $h$ is semipositive.---
If $\tau\in T^{1,0}U\otimes E$ again we write $N(\tau)$ for $N(\tau,\tau)$.

In general, given hermitian forms $M,M'$ on $T^{1,0}U\otimes E$ we write $M\geq M'$ if $M-M'$ is positive semidefinite.
Such forms  can be written
\[
M(\sum_\nu \frac{\partial}{\partial z_\nu}\otimes\xi_\nu)=\sum_{\mu,\nu} \langle M_{\mu\nu}\xi_\mu,\xi_\nu\rangle
\] 
with $M_{\mu\nu}\colon U\to\End H$ (or $M_{\mu\nu}$ sections of $\End E$).
For example, if $M=N$ is the Nakano curvature of the metric $h$ in (3.1), and $R=\sum R_{\mu\nu} dz_\mu\wedge d\overline z_\nu$, then $M_{\mu\nu}=PR_{\mu\nu}$.
Thus $N\ge 0$, or $N$ is semipositive if
\[
\sum_{\mu,\nu} \langle PR_{\mu\nu}\xi_{\mu},\xi_\nu\rangle\geq 0\qquad\text { for arbitrary }\xi_1,\ldots,\xi_m\in H.
\]

\section{Possibly singular hermitian metrics}

In this section we will introduce general hermitian metrics on not necessarily trivial Hilbert bundles, but first we discuss degenerations of norms on an arbitrary complex Banach space $B$.
Let $\|\ \|_1\leq \|\ \|_2\leq\ldots$ be a sequence of norms on $B$, each generating the topology of $B$.
For $x\in B$ set
\[
\|x\|=\lim_j \|x\|_j\leq\infty,\quad\text{ and }A=\{x\in X\colon \|x\|<\infty\}.
\]

\begin{prop}$A\subset B$ is a linear subspace and $(A,\|\ \|)$ is a Banach space.
\end{prop}

\begin{proof}That $A$ is a subspace follows from the triangle inequality, and $\|\ \|$ is clearly a norm on $A$.
To check completeness, consider a Cauchy sequence $x_\nu$ in $(A,\| \ \|)$.
Then $x_\nu$ is Cauchy in $(B, \|\ \|_1)$ as well, hence $x=\lim x_\nu$ exists in the topology of $B$.
For any $j$
\begin{gather*}
\|x-x_\nu\|_j=\lim_{\mu\to\infty} \|x_\mu-x_\nu\|_j\leq\limsup_{\mu\to\infty} \|x_\mu-x_\nu\|,\text{ whence}\\
\|x-x_\nu\|\leq\limsup_{\mu\to\infty} \|x_\mu-x_\nu\|\to 0,\quad\text{ as }\nu\to\infty.
\end{gather*}
Thus $x\in A$ and $x_\nu\to x$ in $\|\ \|$.
\end{proof}

We will apply this construction to the fibers of a holomorphic Hilbert bundle $E\to X$ over a finite dimensional complex manifold.
Recall that a $C^k$ hermitian metric on $E$ is a function $h\colon E\oplus E\to\bC$ that in any local trivialization $E\simeq U\times H$ is a $C^k$ hermitian metric in the sense discussed in Section 3, cf. also [L3]. If $k\ge 2$, the Nakano curvature of $h$ is a hermitian form $N$ on $T^{1,0}X\otimes E$ that can be computed in local trivializations by (3.2), (3.3). Given a continuous hermitian metric $h$ on $E$ and a continuous real $(1,1)$ form $ia$ given in local coordinates as $i\sum a_{\mu\nu}dz_\mu\wedge d\overline z_\nu$, we define a hermitian form $a\otimes h$ on $T^{1,0}X\otimes E$ by
$$ 
(a\otimes h\big)(\sum\frac\partial{\partial z_\nu}\otimes\xi_\nu\big)=\sum a_{\mu\nu}h(\xi_\mu,\xi_\nu).
$$

\begin{defn}For the purposes of this paper a function $h\colon E\to [0,\infty]$ is called a hermitian metric if there is a sequence $h_1\leq h_2\leq\ldots$ of hermitian metrics of class $C^2$ on $E$ such that $h(e)=\lim_j h_j(e)$ for all $e\in E$.
Given a continuous real $(1,1)$ form $ia$ on $X$, we say that the Nakano curvature of $h$ dominates $a$ if the $h_j$ can be chosen to have Nakano curvature $N_j\geq a\otimes h_j$, i.e., for $\sum_\nu t_\nu\otimes\xi_\nu\in T^{1,0} X\otimes E$
\[
N_j (\sum_\nu t_\nu\otimes\xi_\nu)\geq \sum_{\mu,\nu} a(t_\mu,\overline t_\nu) h_j (\xi_\mu,\xi_\nu).
\]
\end{defn}

The definition raises the obvious question whether for the Nakano curvature of  a $C^2$ hermitian metric semipositivity (as in Section 3) is the same as dominating $0$. I do not know if the answer is yes or no, but either outcome would have interesting consequences.
The notion of a possibly singular hermitian metric on a vector bundle of finite rank has already been proposed by de Cataldo, Berndtsson--Pa\v un, and Raufi [BP,dC,R].
Those notions are more general than ours.
At the same time, of these authors only de Cataldo defines what Nakano positivity should mean for such a metric $h$ or what it should mean that the Nakano curvature dominates a form $a$, and in this he is more restrictive than Definition 4.2.
In his definition, like in ours, $h$ should be the increasing limit of $C^2$ hermitian metrics $h_j$ with lower estimates on their Nakano curvature, but he requires additionally that on an open subset of $X$ of full measure the $h_j$ should converge in the $C^2$ topology. (In other respects his definition is less restrictive than ours, namely in what sort of lower estimates are required on the Nakano curvature of $h_j$. We chose the condition in our definition because it is easy to formulate.) 

By Proposition 4.1, a hermitian metric $h$ defines in each fiber $E_x$ a Hilbert space $F_x=\{e\in E_x\colon h(e)<\infty\}$, endowed with the norm $\sqrt{h(e)}$.
The $F_x$ together form a field of Hilbert spaces $F=\coprod_{x\in X} F_x\to X$, see \cite{G}, a notion more general than a Hilbert bundle.

Let $dV$ be a positive, continuous volume form on $X$ and $h, h_j$ as in Definition 4.2.
For any measurable section $f$ of $E$
\[
\int_X h(f) dV=\lim_{j\to\infty} \int h_j (f) dV=\sup_j\int h_j (f) dV,
\]
and we denote by $\|f\|=\|f\|_h\leq\infty$ the square root of this quantity.
We define $L^2(X,h)=L^2 (X,h,dV)$ as the space of $f$ for which $\|f\|<\infty$.
The norm $\|\ \|$ on $L^2(X,h)$ is a Hilbertian norm, and $(L^2 (X,h), \|\ \|)$ is complete, as Proposition 4.1 shows.
Since $\|f\|^2\geq \int_X h_1 (f) dV$, for holomorphic sections convergence in $L^2(X,h)$ implies locally uniform convergence.
Hence holomorphic sections form a closed subspace of $L^2(X,h)$.

Fix now a continuous hermitian metric on $X$ (i.e., on $T^{1,0}X$) with volume form $dV$.
If $\varphi$ is an $E$ valued $(p,q)$ form, we define its norm $|\varphi|=|\varphi|_h\colon X\to [0,\infty]$ in the following way.
Suppose $x\in X$, $z_1,\ldots,z_m$ are local coordinates at $x$, and $\varphi=\sum \varphi_{IJ} dz^I\wedge d\overline z^J$.
Then
\begin{equation}
|\varphi|^2 (x)=|\varphi|^2_h(x)=\sum_{I,J} h(\varphi_{IJ} (x))\leq\infty,
\end{equation}
provided $\partial/\partial z_\nu$ form an orthonormal basis of $T_x^{1,0} X$.
If $\varphi$ is measurable, we put
\begin{equation}
\|\varphi\|=\|\varphi\|_h=\Big(\int_X |\varphi|^2 dV\Big)^{1/2}\leq\infty.
\end{equation}
Clearly
\begin{equation}
\|\varphi\|_h =\lim_{j\to\infty} \|\varphi\|_{h_j}=\sup_j \|\varphi\|_{h_j}.
\end{equation}
We define $L_{pq}^2 (X,h)$ as the space of $\varphi$ such that $\|\varphi\|_h < \infty$; thus $L_{00}^2 (X,h)=L^2 (X,h)$.
Again $L_{pq}^2 (X,h)$ with the norm $\|\ \|_h$ is a Hilbert space.

Consider a Hilbert space $(H, \langle ,\rangle)$ and functions $f,f_j\colon X\to H$.
We say that $f_j\to f$ uniformly weakly if $\langle f_j,v\rangle \to \langle f,v\rangle$ uniformly for every $v\in H$.
For a trivialized Hilbert bundle $E=X\times H\to X$ we define uniform weak convergence of sections accordingly.
Finally, sections $f_j$ of a general Hilbert bundle $E\to X$ converge locally uniformly weakly to a section $f$ if in some neighborhood of each $x\in X$ the bundle has a trivialization in which $f_j\to f$ uniformly weakly.
In general, uniform weak convergence in one trivialization over some $U\subset X$ does not imply uniform weak convergence in some other trivialization, even over relatively compact $V\subset U$, but if the limit section is locally bounded, it does.

\begin{lem}Let $E\to X$ be a holomorphic Hilbert bundle, $h_1\leq h_2\leq\ldots \to h$ hermitian metrics on it, and suppose $\varphi_j\in L^2 (X,h_j)$.
If $\sup_j \|\varphi_j\|_{h_j} < \infty$, and $\varphi_j-\varphi_1$ is holomorphic for each $j$, then a subsequence of $\varphi_j$ will converge locally uniformly weakly to a section $\varphi$ such that $\|\varphi\|_h\leq\sup_j \|\varphi_j\|_{h_j}$ and $\varphi-\varphi_1$ is holomorphic.
\end{lem}

\begin{proof}We can assume $E=X\times H$ is trivialized.
Thus (holomorphic) sections of $E$ are in one--to--one correspondence with (holomorphic) functions $X\to H$, and one can talk about uniform boundedness and equicontinuity of a family of sections.
For example, if we choose a $C^2$ hermitian metric $k\leq h_1$ on $E$, since $\sup_j \|\varphi_j\|_k < \infty$, the holomorphic sections $\varphi_j-\varphi_1$ are locally uniformly bounded and by Cauchy's formula, locally uniformly equicontinuous.
Now closed balls in $H$ are (even sequentially) compact in the weak topology.
Hence for each $x\in X$ the sequence $\varphi_j(x)-\varphi_1(x)$ contains a weakly convergent subsequence.
The Arzel\`a--Ascoli theorem therefore provides a locally uniformly weakly convergent subsequence $\varphi_{j_i}\to\varphi$.
For any $v\in H$ the function $\langle\varphi-\varphi_1,v\rangle=\lim\langle\varphi_{j_i}-\varphi_1,v\rangle$ is holomorphic, whence $\varphi-\varphi_1$ is holomorphic (see e.g.~[M, Exercise 8E], whose solution rests on Cauchy's formula and the principle of uniform boundedness).

To estimate $\|\varphi\|$, let $s=\sup_j \|\varphi_j\|_{h_j}$.
Fix $j$ and choose a sequence $k_1\leq k_2\leq\ldots$ of $C^2$ hermitian metrics that converge to $h_j$.
For any $x\in X$ and $p=1,2,\ldots$, in the Hilbert space $(\{x\}\times H, k_p)$ the sequence $\varphi_{j_i}(x)$ weakly converges to $\varphi(x)$.
Hence
\begin{eqnarray*}
k_p (\varphi(x))&\leq&\liminf_{i\to\infty} k_p (\varphi_{j_i}(x))\quad\text{ and }\\
\|\varphi\|_{k_p}&\leq&\liminf_{i\to\infty} \|\varphi_{j_i}\|_{k_p}\leq s
\end{eqnarray*}
by Fatou's lemma.
Therefore $\|\varphi\|_{h_j}\leq s$ and $\|\varphi\|\leq s$ by (4.3), and the proof is complete.
\end{proof}

\section{The $\overline\partial$--equation and H\"ormander--Skoda type estimates}

Consider a holomorphic Hilbert bundle $E$ over an $m$ dimensional complex manifold $X$.
We denote by $\cD^{pq}(X,E)$ the space of compactly supported smooth $E$ valued $(p,q)$ forms.
The Cauchy--Riemann operator $\overline\partial_E\colon \cD^{pq} (X,E)\to\cD^{p,q+1} (X,E)$ can be defined as in bundles of finite rank, for instance using local trivializations, see e.g.~[M or L1].
It can be extended to an operator defined on a larger subspace of $L^1_{pq,\loc} (X,E)$.
By the latter space we mean the following.
Take a continuous hermitian metric $h_0$ on $E$ and a smooth volume form $dV$ on $X$.
Then a measurable $E$ valued $(p,q)$ form $\varphi$ is in $L^1_{pq,\loc}(X,E)$ if
\[
\int_C |\varphi|_{h_0} dV<\infty\text{ for any compact }C\subset X,
\]
$|\varphi|_{h_0}$ defined in (4.1).
Clearly $L^1_{pq,\loc}(X,E)$ is independent of the choice of $h_0$ and $dV$.

Let ( , ) denote the fiberwise pairing between $E$ and its dual $E^*$.
If $\sum\varphi_{IJ} dz^I\wedge d\overline z^J$ and $\sum \sigma_{KL} dz^K\wedge d\overline z^L$ are local expressions of $E$, respectively $E^*$ valued forms $\varphi,\sigma$, put
\[
(\varphi,\sigma)=\sum_{I,J,K,L}(\varphi_{IJ},\sigma_{KL}) dz^I \wedge d\overline z^J\wedge dz^K\wedge d\overline z^L.
\]
Given $\varphi\in L^1_{pq,\loc} (X,E)$ and $\psi\in L^1_{p,q+1,\loc}(X,E)$, we write $\overline\partial\varphi=\psi$ if for any $\sigma\in\cD^{m-p,m-q-1}(X,E^*)$
\[
\int_X (\varphi,\overline\partial_{E^*}\sigma)=-\int_X (\psi,\sigma).
\]
If such a $\psi$ exists, it is uniquely determined a.e.~by $\varphi$.
Further, if $\varphi\in\cD^{pq}(X,E)$ then $\overline\partial\varphi$ and $\overline\partial_E\varphi$ agree.

In this section we will reproduce the by now standard $L^2$ estimate, essentially due to H\"ormander and Skoda, and streamlined and generalized by Demailly [D1,H,S], for solving $\overline\partial$ in the setting of Hilbert bundles.
Our setting is slightly more general than the one in [D1, VIII.~Theorem 6.1] because we allow bundles $E\to X$ of infinite rank and possibly singular metrics.

Fix a smooth hermitian metric on $X$ and a hermitian metric on $E$.
In addition to the Hilbert spaces $L^2_{pq} (X,h)$ and norms $\|\varphi\|_h$ introduced in Section 4, see (4.2), we will need one more piece of notation.
Consider a continuous $(1,1)$ form $a$ on $X$, $ia\geq 0$.
If $\varphi$ is an $E$ valued $(p,1)$ form, $0\leq p\leq m$, its weighted norm $|\varphi|_{h,a}\colon X\to [0,\infty]$ is defined as follows.
If $x\in X$, choose local coordinates $z_1,\ldots,z_m$ at $x$ in which $a$ diagonalizes:\ $a=\sum a_{\nu\nu} dz_\nu\wedge d\overline z_\nu$ at $x$.
Then
\begin{equation}
|\varphi|^2_{h,a} (x)=\sum_\nu a_{\nu\nu}^{-1}|i_ {\partial/\partial\overline z_\nu}\varphi|_h^2(x)\in [0,\infty],
\end{equation}
where $i_ {\partial/\partial\overline z_\nu}$ stands for contraction.

Recall that a complex manifold $X$ is weakly pseudoconvex if it admits a smooth plurisubharmonic exhaustion function $X\to [0,\infty)$.

\begin{thm}Let $(X,\omega)$ be an $m$ dimensional weakly pseudoconvex K\"ahler manifold, $ia\geq 0$ a continuous $(1,1)$ form on $X$, $E\to X$ a holomorphic Hilbert bundle, and $h$ a hermitian metric on $E$.
Suppose that the Nakano curvature of $h$ dominates $a$, cf.~Definition 4.2.
Given $\psi\in L^1_{m1,\loc}(X,h)$,
\[
\overline\partial\psi=0\qquad\text{ and }\qquad\int_X |\psi|^2_{h,a} \omega^m < \infty,
\]
there exists a $\varphi\in L^2_{m0}(X,h)$ such that
\begin{equation}
\overline\partial\varphi=\psi\qquad\text{ and }\qquad\int_X |\varphi|_h^2 \omega^m\leq\int_X |\psi|^2_{h,a} \omega^m.
\end{equation}
\end{thm}

\begin{proof}First we assume that $h$ is of class $C^2$ and its Nakano curvature satisfies for $\sum t_\nu\otimes\xi_\nu\in T^{1,0} X\otimes E$
\begin{equation}
N(\sum_\nu t_\nu\otimes\xi_\nu)=\sum h\big(R(t_\mu,t_\nu)\xi_\mu,\xi_\nu\big)\geq \sum_{\mu,\nu} a(t_\mu,\overline t_\nu) h(\xi_\mu,\xi_\nu).
\end{equation}
It is straightforward if perhaps tedious to check that the theory expounded in Demailly's book [D1], Chapters VII and VIII is valid in Hilbert bundles of infinite rank.
In particular, Theorem 6.1 in Chapter VIII holds for such bundles.
The hypotheses of that theorem are obviously satisfied now, except possibly the one which involves an operator $A$ on $\bigwedge^{\bullet,\bullet} T^*X\otimes E$, which we will have to check.
Only the action of $A$ on $\bigwedge^{m,1} T^* X\otimes E$ matters here. This was computed in [D1, VII.(7.1)], and
in our notation it can be given as follows.
Let $z_1,\ldots,z_m$ be local coordinates at $x\in X$ such that $\partial/\partial z_\nu$ form a basis in $T_x^{1,0} X$, orthonormal for the inner product induced by $\omega$.
If the curvature of $h$ is $R=\sum R_{\mu\nu} dz_{\mu} \wedge d\overline z_\nu$,
then
\[
h(A\psi,\psi)=\sum_{\mu,\nu} h(R_{\mu\nu} \psi_\mu,\psi_\nu),\quad \psi=\sum \psi_\nu dz\wedge d\overline z_\nu\in\bigwedge^{m,1} T_x^* X\otimes E_x,
\]
where $dz$ stands for $dz_1\wedge\ldots\wedge dz_m$.
If additionally we choose the coordinates so that $a=\sum a_{\nu\nu} dz_\nu\wedge d\overline z_\nu$ at $x$, then by (5.3)
\[
h(A\psi,\psi)\geq\sum_\nu a_{\nu\nu} h(\psi_\nu)=\sum a_{\nu\nu} |i_ {\partial/\partial z_\nu}\psi|_h^2.
\]
Hence if now $\psi$ is an $E$ valued $(m,1)$ form, at $x$
\[
h(A^{-1}\psi,\psi)\leq \sum_\nu a_{\nu\nu}^{-1} |i_ {\partial/\partial\overline z_\nu}\psi|_h (x)=|\psi|_{h,a}^2 (x)
\]
by (5.1).

The estimates in [D1, VIII.~Theorems 6.1 and 4.5] are formulated in terms of $h(A^{-1}\psi,\psi)$, but clearly any greater function will also do.
Replacing $\langle A^{-1}\psi,\psi\rangle=h(A^{-1}\psi,\psi$) in Demailly's formulae by $|\psi|_{h,a}^2$ we obtain Theorem 4.1 when $h$ is of class $C^2$.

A general $h$ is the increasing limit of $C^2$ hermitian metrics $h_j$ whose Nakano curvatures dominate $a$.
By what we have just seen, there are $\varphi_j\in L^2_{m0}(X,h_j)$ such that 
\[
\overline\partial \varphi_j=\psi\quad\text{and}\quad \|\varphi_j\|_{h_j}\leq\int_X |\psi|^2_{h_j,a}\omega^m\leq\int_X |\psi|^2_{h,a}\omega^m.
\]

Denoting the canonical bundle of $X$ by $K$, we can view $\varphi_j$ as sections of $K\otimes E$ and apply Lemma 4.3 with $E$ replaced by $K\otimes E$.
Any subsequential weak limit $\varphi$ will then satisfy (5.2). 
\end{proof}

\section{An extension theorem of Ohsawa-Takegoshi type}

The original publication [OT] of the Ohsawa-Takegoshi extension theorem sparked a lot of interest, various generalizations and alternative approaches have been proposed.
Here we prove an extension for Hilbert bundles following an idea of Bo-Yong Chen [C], see also [Bl].
Guan and Zhou in \cite{GZ1} already proved an extension theorem for finite rank vector bundles.
Undoubtedly their proof could be generalized to Hilbert bundles as well, but Chen's approach is the simplest of all.

At the heart of all extension proofs are estimates for the solution of an equation
\begin{equation}
\overline\partial\varphi=\psi.
\end{equation}
In his estimations Chen is inspired by an idea of Berndtsson that first appeared in \cite{B1} and then in \cite{BC}.
Berndtsson's idea, in a context different from extensions, was as follows.
Given (in \cite{B1,BC} a scalar valued) form $\psi$, suppose we find a solution $\varphi$ whose $L^2$ norm with respect to a certain weight is minimal.
This means that $\varphi$ is orthogonal to Ker $\overline\partial$ in some weighted $L^2$ space.
If $u$ is a bounded function, then $e^u\varphi$ will still be orthogonal to Ker $\overline\partial$, albeit with respect to a modified weight.
So it will be the minimal solution of the ``twisted'' equation
\begin{equation}
\overline\partial (e^u\varphi)=e^u (\overline\partial u\wedge \varphi+u\psi).
\end{equation}
If the weights involved are plurisubharmonic, then one can therefore use H\"ormander's estimate (really, [D1, VIII.~Theorem 6.1]) to bound the solution $e^u\varphi$ of (6.2) in terms of the right hand side.
True, this bound will involve $\varphi$ itself, but if $\overline\partial u$ is sufficiently small, then the bound can be turned into one that involves $\psi$ only, and provides much stronger estimates on $\varphi$ then what follows directly from (6.1).
To carry out this plan, Berndtsson assumed $\sup |\overline\partial u|_{\partial\overline\partial u}<1$.
What Chen noticed was that useful estimates may follow even if $\sup|\overline\partial u|_{\partial\overline\partial u}=1$, and he produced a $u$ that will do the trick for the $\overline\partial$ equation that arises in the extension problem.

Before introducing our version of the extension theorem we have to develop some notation.
Suppose $X$ is a smooth manifold, $Y\subset X$ a smooth submanifold, and $r\colon X\to [0,\infty)$ is a $C^3$ function, $Y=r^{-1}(0)$.
Suppose further that, denoting by dist the distance induced by some Riemannian metric on $X$
\begin{equation}
r(x)\geq c_0\text{ dist}^2 (x,Y),
\end{equation}
$c_0$ a positive constant.
Given such an $r$, any continuous volume form $d\mu$ on $X$ induces a volume form $d\mu_r$ on $Y$ as follows.
Suppose $\theta\in C(Y)$ is compactly supported.
Extend it to a compactly supported $\tilde\theta\in C(X)$.
Then $d\mu_r$ will satisfy
\begin{equation}
\int_Y \theta d\mu_r=\lim_{\varepsilon\to 0}\varepsilon^{-\text{codim}_{\bR}Y}
\int_{\{x\in X\colon r(x) < \varepsilon^2\}} \tilde\theta d\mu.
\end{equation}
Locally $d\mu_r$ can be computed if we introduce local coordinates $x_1,\ldots,x_l$ on $X$ so that $r=\sum_1^k x_i^2$.
If $d\mu=\alpha dx_1\ldots dx_l$ and $c_k$ is the volume of the unit ball in $\bR^k$,
then $d\mu_r=c_k\alpha dx_{k+1}\ldots dx_l$.

In the theorem below we will deal with a hermitian holomorphic Hilbert bundle $(E,h)$ over a K\"ahler manifold $(X,\omega)$.
We write $K$ for the canonical bundle of $X$ and $h^K$ for the metric on $K\otimes E$ induced by $h$ and $\omega$.
That is, if $z_1,\ldots,z_m$ are local coordinates at $x\in X$ such that $\partial/\partial z_\nu$ form an orthonormal basis in $T_x^{1,0}X$, then
\[
h^K (dz_1\wedge \ldots\wedge dz_m\otimes\xi)=h(\xi),\qquad \xi\in E_x.
\]
In other words, if we view a section $g$ of $K\otimes E$ as an $E$ valued $(m,0)$ form (and we will), then with notation (4.1)
\begin{equation}
h^K (g)=|g|_h^2.
\end{equation}

\begin{thm}Let $(X,\omega)$ be a weakly pseudoconvex $m$ dimensional K\"ahler manifold, $Y\subset X$ a complex submanifold of dimension $m-c$, $r\colon X\to [0,1/(2e^2)]$ a $C^3$ function that vanishes on $Y$ and satisfies (6.3).
Define a volume form $d\mu_r$ on $Y$ by (6.4) using the volume form $\omega^m$ on $X$.
Suppose that $\log r$ is plurisubharmonic.
Let furthermore $E\to X$ be a holomorphic Hilbert bundle with a hermitian metric $h$ whose Nakano curvature dominates $0$.
If $f$ is a holomorphic section of $K\otimes E$ over some neighborhood $U\subset X$ of $Y$ and
\[
\int_Y h^K (f) d\mu_r < \infty,
\]
then there is a holomorphic section $g$ of $K\otimes E$ such that $f=g$ on $Y$ and
\begin{equation}
\int_X\  \frac{h^K(g)}{ r^c\log^2 r}\ \omega^m\leq 4^{4+c} \int_Y h^K (f) d\mu_r.
\end{equation}
\end{thm}

\begin{proof}We start by considering a holomorphic Hilbert bundle $F\to X$ with a hermitian metric $k$, an upper semicontinuous $u\colon X\to [-\infty,\infty)$, and the hermitian metric $k'=e^{-u}k$.
When $k,u$ are (finite and) of class $C^2$, one can compute that the curvatures $R,R'$ of $k$ and $k'$ are related by
\[
R'=R+I\partial\overline\partial u,
\]
$I$ denoting the identity endomorphism of $F$.
Hence the Nakano curvatures $N,N'$ satisfy
\begin{eqnarray}
N'(\sum_\nu t_\nu\otimes\xi_\nu)&=&\sum_{\mu,\nu} k' \left( R'(t_\mu,\overline t_\nu)\xi_\mu,\xi_\nu \right)\nonumber\\
&=& e^{-u} N (\sum_\nu t_\nu\otimes\xi_\nu)+\sum_{\mu,\nu}\partial\overline\partial u(t_\mu,\overline t_\nu) k'(\xi_\mu,\xi_\nu).
\end{eqnarray}
It follows that if the Nakano curvature of $k$ is semipositive and $u$ is plurisubharmonic, then the Nakano curvature of $k'$ is also semipositive.
By approximation, this will also hold for possibly singular metrics $k$ and plurisubharmonic functions $u$ that are decreasing limits of $C^2$ plurisubharmonic functions.
(6.7) also shows that if the Nakano curvature of a hermitian metric $k$ dominates $0$ and $u\in C^2(X)$, then the Nakano curvature of $k'$ will dominate $\partial\overline\partial u$.

Now let us put ourselves in the setting of Theorem 6.1.
All integrals of functions on $X$ and on open subsets of $X$ will be with respect to the volume form $\omega^m$; for brevity, from now on we will omit the volume form from the integrals.
Since $h$ can be approximated from below by $C^2$ hermitian metrics, we can assume $h$ is already such.
We will also assume that $X$ is a smoothly bounded, relatively compact open subset of a complex manifold $X_0$, that $\omega$ extends to a smooth K\"ahler form $\omega_0$ on $X_0$, that $Y=Y_0\cap X$, where $Y_0\subset X_0$ is a complex submanifold intersecting $\partial X$ transversely; that $(E,h)$ extends to a holomorphic hermitian Hilbert bundle (also denoted $(E,h)$) over $X_0$, and that $f$ extends to a holomorphic section of $E\otimes K_{X_0}$ over a neighborhood $U_0\subset X_0$ of $Y_0$.
Once this special case is handled, the general case will follow.
We would take a smooth plurisubharmonic exhaustion function $\rho\colon X\to [0,\infty)$, solve the extension problem on generic sublevel sets of $\rho$ and apply Lemma 4.3 combined with a diagonal selection procedure.

With all the extra assumptions above, we fix a smooth function $\chi\colon [0,\infty)\to [0,1]$,
\begin{equation*}
\chi(t)=
\begin{cases}
1,&\text{if $t<1/4$}\\
0,&\text{if $t>1$}
\end{cases}
, \qquad |\chi'(t)|\leq 2 \text{ for all } t.
\end{equation*}
Set, for $\var>0$,
\[
\Omega_\var=\{x\in X\colon r(x) < \var^2\}.
\]
When $\var$ is sufficiently small, $\Omega_\var\subset U_0$, so that 
\[
f'=f'_\var=\begin{cases}  \chi(r/\var^2)f&\text{on $\Omega_\var$}\\
0&\text{on $X\setminus \Omega_\var$}\end{cases}
\]
is a smooth section of $K\otimes E$. Let $\psi=\overline\partial f'$, a smooth closed $K\otimes E$ valued $(0,1)$ form.
Presently we will check that the equation $\overline\partial\varphi=\psi$ has a solution $\varphi=\varphi_\var\in L^2 (X, h^K/r^c)$.
Accepting this for the moment, $g=f'-\varphi\in\Gamma (X,K\otimes E)$. Hence $\varphi$ is smooth and
\begin{equation}
\varphi|Y=0,
\end{equation}
because $1/r^c$ is nowhere integrable at $Y$.
Thus $g$  agrees with $f$ on $Y$, and what we need to do is estimate it.
Since we can freely add to $\varphi$ any holomorphic section in $L^2(X,h^K/r^c)$, we can arrange that $\varphi$ is orthogonal to the closed subspace of holomorphic sections in $L^2(X,h^K/r^c)$.
It follows that with any $u\in C^2(\overline X)$ the section $\theta=e^u\varphi$ is orthogonal to holomorphic sections in $L^2(X,e^{-u} h^K/r^c)$.---This latter space is the same as $L^2(X,h^K/r^c)$ but their inner products are different.---Therefore $\theta$ is the solution of
\begin{equation}
\overline\partial\theta=\overline\partial (e^u\varphi)=e^u (\varphi\overline\partial u+\psi)
\end{equation}
that has minimal norm in $L^2(X,e^{-u} h^K/r^c)$.

Let us abbreviate $e^{-u} h/r^c=k$.
As we observed above, the Nakano curvature of $h/r^c=e^{-c\log r}h$ dominates $0$ (note that $\log(r+1/j)\searrow\log r$), and so the Nakano curvature of $k$ dominates $\partial\overline\partial u$.
We will choose a plurisubharmonic $u$ so that
\begin{equation}
\int_X |\varphi\overline\partial u|^2_{k^K,\partial\overline\partial u}\quad ,\quad \int_X |\psi|^2_{k^K,\partial\overline\partial u}<\infty.
\end{equation}
Viewing $\varphi\overline\partial u$ and $\psi$ as $E$ valued $(m,1)$ forms, the integrands above are the same as 
$|\varphi\overline\partial u|^2_{k,\partial\overline\partial u}$ and $|\psi|^2_{k,\partial\overline\partial u}$, hence by Theorem 5.1 and (6.9) we would conclude
\begin{gather*}
\int_X k^K(\theta)\leq\int_X |e^u (\varphi\overline\partial u+\psi)|^2_{k^K,\partial\overline\partial u},\quad\text{ or}\\
\begin{align*}
\int_X \frac{e^u}{ r^c} h^K (\varphi)\leq \int_{X\backslash\Omega_\var}&\frac{e^u}{r^c}|\varphi\overline\partial u|^2_{h^K,\partial\overline\partial u}\\
&+2 \int_{\Omega_\var}\frac{e^u}{ r^c} |\varphi\overline\partial u|^2_{h^K,\partial\overline\partial u}
+2\int_{\Omega_\var}\ \frac{e^u}{ r^c}|\psi|^2_{h^K,\partial\overline\partial u}
\end{align*}\end{gather*}
Since $|\varphi\overline\partial u|^2_{h^K,\partial\overline\partial u}=h^K(\varphi)|\overline\partial u|^2_{\partial\overline\partial u}$, putting
\begin{equation}
\lambda(x)=\begin{cases}
1-|\overline\partial u|^2_{\partial\overline\partial u}(x),&\text{if $x\in X\backslash\Omega_\var$}\\
1-2|\overline\partial u|^2_{\partial\overline\partial u}(x),&\text{if $x\in\Omega_\var$}\end{cases},
\end{equation}
we obtain
\begin{equation}
\int_X \lambda \ \frac{e^u}{ r^c}\ h^K (\varphi)\leq 2 \int_X\ \frac{e^u}{ r^c} |\psi|^2_{h^K,\partial\overline\partial u}.
\end{equation}

Remains to prove that $\overline\partial\varphi=\psi$ indeed has a solution in $L^2(X,h^K/r^c)$ and to choose $u$ so that (6.10) holds and (6.12) implies the estimate (6.6) for $g=f'-\varphi$.

Following Chen's idea in a similar setting, we make sure that $\var<1/(2e)$ and let 
\[
\rho=-\log (r+\var^2),\qquad \eta=\rho+\log\rho,\quad\text{and}\quad u=-\log\eta.
\]
Thus $-\rho,-\eta$, and $u$ are plurisubharmonic and $C^3$,
\[
2<\rho<\eta<2\rho\qquad\text{and }\qquad u<0.
\]
We need to estimate $|\overline\partial u|_{\partial\overline\partial u}$. This will take  a little bit of computation:
\begin{eqnarray}
\overline\partial\rho&=&-\frac{\overline\partial r}{r+\var^2},\quad \partial\overline\partial\rho=\frac{\partial r\wedge\overline\partial r}{(r+\var^2)^2}-\frac{\partial\overline\partial r}{r+\var^2},\\
\eta\overline\partial u&=&-\overline\partial\eta=-\big(1+\frac1\rho\big)\overline\partial\rho=
\big(1+\frac1\rho\big)\frac{\overline\partial r}{r+\var^2},\\
 i\eta^2\partial\overline\partial u&=&i\Big( \big(1+\frac1{\rho}\big)^2+\frac{\eta}{\rho^2}\Big)\partial\rho\wedge\overline\partial\rho-i\eta\big(1+\frac 1{\rho}\big)\partial\overline\partial\rho\\
&\geq& i\Big( \big(1+\frac 1{\rho}\big)^2+\frac 1{\rho}\Big) \partial\rho\wedge\overline\partial\rho.\nonumber
\end{eqnarray}

From (6.13), (6.15), on $X$
\begin{equation}
i\partial r\wedge\overline\partial r\leq \frac{i\eta^2(r+\var^2)^2}{ (1+1/\rho)^2+1/\rho}\ \partial\overline\partial u,\quad
|\overline\partial r|^2_{\partial\overline\partial u}\leq\frac {\eta^2(r+\var^2)^2}{ (1+1/\rho)^2+1/\rho}.
\end{equation}
However, on $\Omega_\var$ the estimate can be improved.
Indeed, (6.13) gives, when $r<\var^2$
\begin{align*}
(r+\var^2)^2 (-i\partial\overline\partial\rho-i\partial\rho\wedge\partial\overline\rho)=i(r+\var^2)\partial\overline\partial r&-2i\partial r\wedge\overline\partial r\\
&\geq2i (r\partial\overline\partial r-\partial r\wedge\overline\partial r)\geq 0,
\end{align*}
the latter simply expressing that $i\partial\overline\partial\log r\geq 0$.
Hence by (6.13), (6.15)
\[
i\partial r\wedge\overline\partial r=i(r+\var^2)^2 \partial\rho\wedge\overline\partial\rho\leq -4i\var^4 \partial\overline\partial\rho\leq 4 i\var^4 \eta\partial\overline\partial u,
\]
so that $|\overline\partial r|^2_{\partial\overline\partial u}\leq 4\var^4\eta$.
Therefore
\begin{equation}
|\overline\partial u|^2_{\partial\overline\partial u}=\frac{(1+1/\rho)^2}{\eta^2(r+\var^2)^2}\ |\overline\partial r|^2_{\partial\overline\partial u}\leq \frac {16}{\eta}\quad\text{on }\Omega_\var,
\end{equation}
in view of (6.14).
On supp$\ \psi\subset \Omega_\var\backslash \Omega_{\var/2}$ we can estimate
\begin{equation*}
|\psi|^2_{h^K,\partial\overline\partial u}= |f\chi' (r/\var^2) \overline\partial r/\var^2|^2_{h^K,\partial\overline\partial u}
\leq 4 h^K (f)|\overline\partial r|^2_{\partial\overline\partial u}/\var^4\leq 16 \eta h^K (f),
\end{equation*}
and so 
\begin{equation}
\int_X\ \frac{e^u}{ r^c}\ |\psi|^2_{h^K,\partial\overline\partial u}\leq \frac{4^{2+c}}{\var^{2c}}\int_{\Omega_\var\backslash \Omega_{\var/2}} h^K(f)<\infty.
\end{equation}
We apply Theorem 5.1 with the hermitian metric $e^{-u}h/r^c$ to obtain a solution $\varphi\in L^2(X,e^{-u} h^K/r^c)=L^2(X,h^K/r^c)$ of the equation $\overline\partial\varphi=\psi$, which we choose to have the minimal norm in $L^2(X,h^K/r^c)$.

Comparing (6.14) and (6.16) we find
\begin{equation}
|\overline\partial u|^2_{\partial\overline\partial u}= \frac{(1+1/\rho)^2}{ \eta^2(r+\var^2)^2} |\overline\partial r|^2_{\partial\overline\partial u} \leq \frac{(\rho+1)^2}{ (\rho+1)^2+\rho} \leq  1-\frac 1 {4\rho},
\end{equation}
as $\rho>2$; so that
\begin{equation*}
\int_X \frac{e^u}{ r^c} |\varphi\overline\partial u|_{h^K,\partial\overline\partial u}=\int_X \frac{e^u}{ r^c} h^K (\varphi)|\overline\partial u|_{\partial\overline\partial u} \leq  \int_X \ \frac{e^u}{ r^c} \ h^K (\varphi)<\infty.
\end{equation*}
This latter and (6.18) prove (6.10), and it follows that $\varphi$ satisfies (6.12).
Looking up the definition of $\lambda$, (6.11), and comparing it with (6.17), (6.19) we obtain
\[
\lambda\geq\begin{cases}
1/4\rho&\text{on $X\backslash \Omega_\var$}\\
1/2&\text{on $\Omega_\var$}\end{cases}\ \geq \frac 1{ 4\rho}\geq \frac 1 {4|\log r|}
\]
when $\var>0$ is sufficiently small.
We also note that $e^u=1/\eta\geq 1/(2\log r)$.
Putting all this in (6.12) and taking (6.18) into account estimates $\varphi=\varphi_\var$:
\begin{equation}
\int_X\ \frac{h^K(\varphi_\var)}{ r^c\log^2 r}\leq \frac{4^{4+c}}{\var^{2c}}\ \int_{\Omega_\var}\ h^K (f).
\end{equation}

Now we let $\var\to 0$.
To estimate the limit on the right, set $\Omega_\var^0=\{x\in X_0\colon r(x) < \var^2\}$ and choose a compactly supported continuous function $\sigma\colon X_0\to [0,1]$ such that $\sigma|X\equiv 1$.
Then
\[
\limsup_{\var\to 0} \var^{-2c}\int_{\Omega_\var} h^K (f)\leq\limsup_{\var\to 0} \var^{-2c} \int_{\Omega^0_\var}\sigma h^K (f)=\int_{Y_0} \sigma h^K (f) d\mu_r.
\]
Approximating the characteristic function of $\overline X$ by such $\sigma$, this last integral gets as close to $\int_Y h^K(f) d\mu_r$ as we please, whence by (6.20)
\[
\limsup_{\var\to 0}\ \int_X \ \frac{h^K(\varphi_\var)}{ r^c\log^2 r}\leq 4^{4+c}\int_Y h^K (f) d\mu_r.
\]
At the same time, since $1/(r^c\log^2 r)$ is integrable on $X$,
\[
\lim_{\var\to 0}\ \int_X  \frac{h^K(f'_\var)}{ r^c\log^2 r}=\lim_{\var\to 0}\ \int_{\Omega_\var} \frac{h^K\left(\chi (r/\var^2) f\right)}{ r^c\log^2 r}=0.
\]
Therefore $g_\var= f'_\var-\varphi_\var$ is holomorphic, $g_\var |Y=f|Y$, and
\[
\limsup_{\var\to 0}\ \int \ \frac{h^K(g_\var)}{ r^c\log^2 r}\leq 4^{4+c}\int_Y h^K (f) d\mu_r.
\]
By  Lemma 4.3 a subsequential weak limit $g$ of $g_\var$ will then satisfy the requirements.
\end{proof}

\section{The proof of Theorem 1.2}

Let $U\subset\bC^m$ be open and $E\to U$ a holomorphic Hilbert bundle as in Theorem 1.2.
Let furthermore $h_1\geq h_2\geq\ldots$ be hermitian metrics with Nakano curvatures dominating $0$, and assume that $h=\lim_j h_j$ is  bounded below by a continuous hermitian metric.
Much like in Section 2, if $P\subset\bC^m$ is a complex affine hyperplane, $W\subset P\cap U$ is relatively open, and $g$ is a measurable section of $E|W$, we define $\|g\|\in [0,\infty]$ by
\begin{eqnarray}
\|g\|^2&=&\inf_j \int_W h_j (g),\quad\text{so that}\nonumber\\
\|g\|^2&=&\begin{cases} \infty\qquad\text{or}\\ \lim_j \int_W h_j (g)=\int_W h(g).\end{cases}
\end{eqnarray}
Let $x\in U$.
First we prove the following characterization of $M=M_x=\bigcup_j E (h_j,x)$.

\begin{lem}Suppose $M$ is finitely generated as an $\cO_x$-module.
The germ $\mathbf{f}_x$ of an $f\in\Gamma(E)$ belongs to $M$ if and only if for any sufficiently small neighborhood $V\subset U$ of $x$ and any hyperplane $P_0\subset\bC^m$
\begin{equation}
\liminf{\rm{dist}}(x,P)\|f|V\cap P\|=0,\quad\text{as }P\|P_0\text{ and }{\rm{dist}}(x,P)\to 0.
\end{equation}
\end{lem}

We need the following simple

\begin{prop}Let $W$ be a complex vector space, $(B,\|\ \|)$ a normed space, $L\colon W\to B$ and $l\colon W\to\bC$ linear.
If $|l(w)|\leq C\|L(w)\|$ for every $w\in W$ with some constant $C$, then there is a linear map $a\colon B\to\bC$ of norm $\leq C$ such that $l=aL$.
\end{prop}

\begin{proof}First we define $a$ on $L(W)\subset B$.
If $u=L(w)\in L(W)$, set $a(u)=l(w)$.
This is independent of the choice of $w$, since Ker $L\subset\text{ Ker }l$.
Further,
\[
|a(u)|= |l(w)|\leq C\|L(w)\|=C\|u\|.
\]
By the Banach--Hahn theorem we extend $a$ to a linear form on $B$ satisfying the same estimate; this extension will clearly do.
\end{proof}

\noindent
\begin{proof}[Proof of Lemma 7.1]
We will assume $x=0$.
The ``only if'' direction follows from Fubini's theorem as in the proof of Lemma 2.1.
Conversely, we will show that if $\mathbf f_0\not\in M$ then for any neighborhood $V\subset U$ of 0 and for some hyperplane $P_0$
\begin{equation}
\liminf {\rm{dist}} (x,P)\| f|V\cap P\|>0,\quad \text{as }P\|P_0\text{ and }{{\rm{dist}}}(x,P)\to 0.
\end{equation}

Fix $V$.
We can assume it is pseudoconvex, relatively compact in $U$, and there are $g^1,\ldots,g^p\in\Gamma (V,E)$ whose germs generate $M$.
We can also assume $\int_V h_{j_0} (g^i)<\infty$ with some $j_0$ and $i=1,\ldots,p$.

Write $\pi\colon E\to U$ for the bundle projection.
If $\Delta\subset\bC$ is the unit disc and $\alpha\colon\Delta\to E^*|V$ is holomorphic, mapping $0\in\Delta$ to the zero vector in $E_0^*$, we can associate with sections $g\in\Gamma(V,E)$ functions $\alpha^\star g\in\cO(\Delta)$ by evaluating $\alpha(s)$ on $g(\pi\alpha(s))$:
\[
(\alpha^\star g)(s)=\alpha (s) g(\pi\alpha (s)).
\]
Likewise we can pull back germs $\mathbf g_0$ of sections of $E$ to germs $\alpha^\star \mathbf g_0\in\cO_{(\bC,0)}$. Set $\alpha^\star M=\{\alpha^\star \mathbf g_0\colon \mathbf g_0\in M\}$.
We claim that there is an $\alpha$ such that
\begin{equation}
\alpha^\star \mathbf f_0\not\in \cO_{(\bC,0)}\alpha^\star M.
\end{equation}

Indeed, [L2, Lemma 5.1] implies that over some neighborhood of $0\in U$, our $f,g^1,\ldots,g^p$ are in fact sections of a finite rank holomorphic subbundle of $E$.
We will construct $\alpha$ as the pull back of a map into the dual of this subbundle, and so we can assume at this juncture that the rank of $E$ itself is finite.
Of course, we can also assume $E$ is trivial, $E=U\times\bC^n\to U$.
With any $g\in\Gamma(V,E)$ given by $g(z)=(z,g_1(z),\ldots,g_n(z))$ we associate a function $\hat g\in\cO (V\times\bC^n)$,
\[
\hat g(z,w)=\sum_\nu g_\nu (z) w_\nu.
\]
We do likewise with germs of sections of $E$ at $0\in U$.
Consider the integral closure $\hat M\subset\cO_{(\bC^{m+n},0)}$ of the ideal generated by $\hat{\mathbf g}_0^1,\ldots,\hat{\mathbf g}_0^p$.
Thus, again by [LT, Th\'eor\`eme 2.1], $\hat M$ consists of germs $\boldsymbol\varphi_0\in\cO_{(\bC^{m+n},0)}$ of functions $\varphi$ that satisfy
\begin{equation}
|\varphi|^2\leq C^2 \left( |\hat g^1|^2+\ldots + |\hat g^p|^2\right)
\end{equation}
on some neighborhood of $0\in\bC^{m+n}$, with some constant $C$.
Suppose $\varphi=\hat g$ satisfies (7.5) on some neighborhood; the neighborhood can be taken of form $\pi^{-1}(V_0)$, with $V_0\subset V$ a neighborhood of $0\in\bC^{m}$. Let $z\in V_0$.
We apply Proposition 7.2 with $W=\bC^n$, $B$ the Euclidean space $\bC^p$, the components of $L:\bC^n\to\bC^p$ the functions $\hat g^i(z,\cdot)$, $i=1,\ldots,p$, and $l=\hat g(z,\cdot)$.
This produces $a_i\in\bC$ such that $g(z)=\sum a_i g^i(z)$ and $\sum|a_i|^2\leq C^2$.
By Schwarz's inequality 
\[
h_{j_0} (g(z))\leq \sum_i |a_i|^2\sum_i h_{j_0} (g^i (z))\leq C^2\sum_i h_{j_0} (g^i (z)).
\]
We see that for a $\varphi$ of form $\hat g$  (7.5) implies $\mathbf g_0\in M$.
Since our $\mathbf f_0\not\in M$, it follows that $\varphi=\hat f$ does not satisfy (7.5) on any neighborhood of $0$, i.e.~$\hat{\mathbf f}_0\not\in\hat M$.
As in the proof of Lemma 2.1, according to [LT] this implies that there is a holomorphic $\alpha\colon\Delta\to V\times\bC^n$, $\alpha(0)=0$, such that $\hat{\mathbf f}_0\circ\alpha\not\in\cO_{(\bC,0)}\hat M\circ\alpha$.
Then $\alpha$, viewed as a map $\Delta\to E^*|V=V\times\bC^n$, satisfies (7.4).

There is quite some flexibility in the choice of $\alpha$.
To wit, $\hat{\mathbf f}_0\circ\alpha\not\in \cO_{(\bC,0)}\hat M\circ\alpha$ means that at $0\in\Delta$ the order of each $\hat g_i\circ\alpha$ is greater than the order of $\hat f\circ\alpha$.
This will clearly persist if we perturb $\alpha$ by adding terms whose order is greater than the order of $\hat f\circ\alpha$.
We use this flexibility to arrange that $\pi\circ\alpha\neq 0$.

Next choose a hyperplane $P_0$ through $0\in\bC^m$ that does not contain $\pi\alpha(\Delta)$.
Again we can adjust coordinates in $\bC$ and in $\bC^n$ so that $P_0=\{z\in\bC^m\colon z_1=0\}$, that $\alpha$ is holomorphic in a neighborhood of $\overline\Delta$, that $F=\alpha^\star f\neq 0$ on $\overline\Delta\backslash \{0\}$, that the first component of $\pi\alpha$ satisfies with some $k\in\bN$ 
\[
\pi_1\alpha(s)=s^k,\quad s\in\Delta,
\]
and $\alpha(\overline\Delta)\subset V\times\bC^n$.
This latter implies for any $g\in\Gamma(V,E)$
\begin{equation}
\max_{\overline\Delta} |\alpha^\star g|^2\leq C_2^2 \int_V h(g)\leq C_2^2 \int_V h_j (g),\qquad j=1,2\ldots,
\end{equation}
with some $C_2$ independent of $g$.

Let $\sigma\in\Delta\backslash \{0\}$ and let $P_\sigma=\{z\in\bC^m\colon z_1=\sigma\}$.
Assume first $\|f|V\cap P_\sigma\|<\infty$.
We apply Theorem 6.1 with $(X,\omega)=(V,\sum dz_\nu\wedge d\overline z_\nu)$, $Y=V\cap P_\sigma$ and $r(z)=c |z_1-\sigma|^2$.
If the constant $c>0$ is sufficiently small, then $r\leq 1/(2e^2)$ on $V$.
The volume form $d\mu_r$ on $V\cap P_\sigma$ is a constant multiple of the Euclidean volume.
Choose $j$ such that
\[
\int_{V\cap P_\sigma} h_j (f)\leq 2\|f|V\cap P_\sigma\|^2.
\]
Since the bundles $(K\otimes E,h_j^K)$ and $(E,h_j)$ are isometrically isomorphic, Theorem 6.1 produces a $g\in\Gamma (V,E)$ with
\begin{equation}
f=g\text{ on }V\cap P_\sigma\qquad\text{ and }\qquad\int_V h_j(g)\leq C_3^2 \|f|V\cap P_\sigma\|^2,
\end{equation}
where $C_3$ is independent of $\sigma$.
Thus $\mathbf g_0\in M$ and the germ of $G=\alpha^\star g$ is in $\alpha^\star M$.
As the germ of $F=\alpha^\star g$ is not in $\cO_{(\bC,0)}\alpha^\star M$, it follows that $G=o(F)$ at $0\in\Delta$.
Further, by (7.7) $F(\root k\of \sigma)=G(\root k\of \sigma)$ for any choice of $k$'th root $\root k\of \sigma$.
Hence by Proposition 2.2
\begin{equation}
\max_{\overline\Delta} |\alpha^\star g|=\max_{\overline\Delta} |G|\geq C_1/|\sigma|.
\end{equation}
(7.6), (7.7), and (7.8) together yield
\[
\|f|V\cap P_\sigma\|\geq \frac{C_1}{ C_2 C_3 |\sigma|},\quad \sigma\in\Delta\backslash \{0\}.
\]
As this also holds when $\|f|V\cap P_\sigma\|=\infty$, (7.3) has been proved and, with it, Lemma 7.1.
\end{proof}

\noindent
\begin{proof}[Proof of Theorem 1.2.]
We prove by induction on $m$.
Again, the case $m=0$ is obvious.
Assume the statement for $m-1$, and consider the $m$ dimensional theorem.
To apply the induction hypothesis we have to understand whether the restrictions $E|P$ to hyperplanes $P\subset\bC^m$ satisfy the hypothesis of the theorem.
On the one hand, if rk $E<\infty$ then of course rk $E|P<\infty$.
On the other hand, if $\bigcup_j\cE (h_j)$ is locally finitely generated, then the question becomes whether the sheaf ${\cF}\to P$ whose stalk at $x\in P$ is
\begin{equation}
\{\boldsymbol\varphi_x\colon \varphi\in\Gamma (V\cap P,E),\int_{V\cap P} h_j (\varphi) < \infty\text{ for some } j\text{ and open }V\subset U,\ x\in V \}
\end{equation}
is also locally finitely generated.
This will be true for almost all $P$.
For, at the price of shrinking
$U$, we can assume there are $g^1,\ldots,g^p\in\Gamma(E)$ that generate each stalk of $\bigcup_j \cE(h_j)$.
We claim that whenever $P$ is such that
for some $j$
\[
\int_{U\cap P} h_{j} (g^i)<\infty,\qquad i=1,\ldots,p,
\]
the sheaf $\cF$ is generated by $g^i|P$.
Indeed, let $\boldsymbol\varphi_x\in{\cF}_x$ be the germ of a $\varphi\in
\Gamma (V\cap P,E)$ as in (7.9).
Assuming, as we may, that $V$ is pseudoconvex, Theorem 6.1 can be applied as in the proof of Lemma 7.1 to extend $\varphi$ to a $g\in\Gamma (V,E)$ such that $\int_V h_j (g)<\infty$.
Hence $g_x\in \bigcup_j E(h_j,x)=M_x$.
But since the latter module is generated by the $g^i$, it follows that $\boldsymbol\varphi_x=\mathbf g_x|P$ is in the module generated by $\mathbf g_x^i|P$.

We need to show $M_x=E(h,x)$ for arbitrary $x$, which we will take to be 0. As above, we assume that $g^1,\ldots,g^p\in\Gamma(E)$ generate each stalk of $\cE$.
Fix a relatively compact neighborhood $V_0\subset U$ of 0 and an $f\in\Gamma (V_0,E)$ such that $\int_{V_0} h(f)<\infty$.
We will show that $\mathbf f_0\in M_0$ using the characterization in Lemma 7.1.
Take a neighborhood $V$ of 0, relatively compact in $V_0$, and a hyperplane $P_0$.
Again we assume $P_0=\{t\in\bC^m\colon z_1=0\}$ and for $s\in\bC$ write $P_s=\{z\in\bC^m\colon z_1=s\}$.
Fubini's theorem guarantees that there are a $j_0$ and a set $S\subset\bC$ of full measure such that
\[
\int_{V_0\cap P_s} h(f)<\infty\quad\text{ and }\quad\int_{V_0\cap P_s} h_{j_0} (g^i) < \infty\text{ for }s\in S,\ i=1,\ldots,p,
\]
and
\begin{equation}
\liminf_{S\ni s\to 0} |s|^2 \int_{V_0\cap P_s} h(f)=0.
\end{equation}
The induction hypothesis implies that for each $s\in S$ there is a $j$ such that $\int_{V\cap P_s} h_j (f)<\infty$, whence 
\[
\| f|V\cap P_s\|^2=\int_{V\cap P_s} h(f),\qquad\text{ cf.~(7.1)}.
\]
Hence (7.10) implies (7.2) (with $x=0$), and so by Lemma 7.1, $\mathbf f_0\in M_0$ indeed.
\end{proof}

Here is an example that shows that in Theorem 1.2 the condition of finite generation can not be simply dropped.
Let $U\subset\bC$ be the unit disc and $E$ the trivial bundle $U\times l^2\to U$.
We endow $E$ with metrics that are determined by a sequence $\sigma=(\sigma_1,\sigma_2,\ldots)$ of nonnegative numbers,
\[
h(z,w)=h^\sigma (z,w)=\sum_\nu |w_\nu|^2/|z|^{2\sigma_\nu},\quad z\in U,\ w=(w_\nu)\in l^2.
\]
We only consider $\sigma_\nu$ such that $\sup_\nu \sigma_\nu <2$.
Let $V\subset U$ be a neighborhood of 0.
If $f\in\Gamma (V,E)$ given by $f(z)=(z,f_1(z),f_2(z),\ldots)$ is in $L^2_{\loc}(V,h)$, then $f_\nu(0)=0$ whenever $\sigma_\nu\geq 1$.
Assuming $V$ is the disc $\{z\in\bC\colon |z| <\rho\}$, with $\rho\in (0,1]$, the $L^2(V,h)$ norm of such $f$ is given in terms of the Taylor coefficients of $f_\nu(z)=\sum_k a_{\nu k} z^k$:
\[
\int_V |f|_h^2=\pi\sum_{\nu,k} |a_{\nu k}|^2\ \frac{\rho^{2(k-\sigma_\nu+1)}}{ k-\sigma_\nu+1}.
\]

Let $\sigma_\nu=1-1/\nu$.
Then $f_\nu (z)\equiv 1/\nu^2$ defines a section $f$ of $E$ whose germ $\mathbf f_0\in E(h^\sigma)$.
However, for no $j\in\bN$ is this germ in $E(h^{(1+1/j)\sigma})$, and so
\[
\bigcup^\infty_{j=1} (E (h^{(1+1/j)\sigma},0)\subsetneq E(h^\sigma,0).
\]

At the same time, any holomorphic section in $L^2(V,h^\sigma)$ can be approximated by holomorphic sections in $\bigcup_j L^2 (V,h^{(1+1/j)\sigma})$.
To end this paper, we propose the following float.
(According to the late Lee Rubel, a float, much like a conjecture, is a mathematical statement one would like to see proved; but while to make a conjecture the conjecturer should have substantial evidence in its favor, a float is allowed once it occurs to the floater that the statement might be true.)

Consider a holomorphic Hilbert bundle $E\to U$ over a pseudoconvex open $U\subset\bC^m$ and let $h_1\geq h_2\geq\ldots$ be hermitian metrics on $E$.
Assume that the Nakano curvature of each $h_j$ dominates $0$ and that $h=\lim h_j$ dominates a continuous hermitian metric.
Let $f\in\Gamma (E)\cap L^2 (U,h)$ and $V\subset U$ be a sublevel set of a plurisubharmonic exhaustion function.
Then in the Hilbert space $L^2(V,h)$ the section $f|V$ can be approximated by holomorphic sections of $E|V$ that are in $\bigcup_j L^2 (V,h_j)$.

\end{document}